\documentclass[oneside,english,a4wide]{amsart}
\usepackage[latin9]{inputenc}
\usepackage[a4paper]{geometry}
\usepackage{mathabx}
\usepackage{esvect}
\usepackage{babel}
\usepackage{mathtools}
\usepackage{amstext}
\usepackage{amsthm}
\usepackage{amssymb}
\usepackage{enumerate}
\usepackage{esint}
\usepackage{graphicx}
\usepackage{bbm}
\usepackage[all]{xy}
\usepackage{mathrsfs}
\usepackage[unicode=true,
bookmarks=true,bookmarksnumbered=true,bookmarksopen=true,bookmarksopenlevel=2,
breaklinks=false,pdfborder={0 0 1},backref=false,colorlinks=false]
{hyperref}
\hypersetup{
	colorlinks,linkcolor=blue,anchorcolor=blue,citecolor=blue}
\usepackage{breakurl}
\usepackage{autobreak}

\usepackage[svgnames]{xcolor} 

\usepackage{graphicx} 

\makeatletter
\numberwithin{equation}{section}
\numberwithin{figure}{section}
\theoremstyle{plain}
\newtheorem{thm}{\protect\theoremname}[section]
\theoremstyle{plain}

\newtheorem{corollary}[thm]{Corollary}
\theoremstyle{definition}

\theoremstyle{plain}
\newtheorem{lem}[thm]{\protect\lemmaname}
\newtheorem{lemma}[thm]{Lemma}
\theoremstyle{remark}
\newtheorem{rem}[thm]{\protect\remarkname}

\theoremstyle{definition}

\newtheorem*{claim*}{Claim}

\theoremstyle{remark}

\theoremstyle{definition}
\newtheorem*{defn*}{\protect\definitionname}

\usepackage{babel}
\usepackage{amscd}

\makeatother

\providecommand{\definitionname}{Definition}
\providecommand{\lemmaname}{Lemma}
\providecommand{\propositionname}{Proposition}
\providecommand{\remarkname}{Remark}
\providecommand{\theoremname}{Theorem}

\newcommand{\Rmnum}[1]{\expandafter\@slowromancap\romannumeral #1@}
\makeatother

\newcommand{\M}{{\mathcal M}}

\newcommand{\sfT}{{\mathsf T}}

\newcommand{\8}{\infty}

\newcommand{\la}{\langle}
\newcommand{\ra}{\rangle}

\newcommand{\be}{\begin{eqnarray*}}
	\newcommand{\ee}{\end{eqnarray*}}
\newcommand{\beq}{\begin{equation}}
\newcommand{\eeq}{\end{equation}}
\newcommand{\beqn}{\begin{equation*}}
\newcommand{\eeqn}{\end{equation*}}
\newcommand{\bs}{\begin{split}}
	\newcommand{\es}{\end{split}}

\renewcommand{\a}{\alpha}

\newcommand\norm[1]{ \left\| #1 \right\| }
\newcommand\jdz[1]{ \left| #1 \right| }
\newcommand\sk[1]{ \left( #1 \right) }
\newcommand\mk[1]{ \left[ #1 \right] }

\newcommand\pdt{ \frac{\partial }{\partial t} }
\renewcommand\d{{\rm d}}

\begin{document}

	
	
\title[From the Littlewood-Paley-Stein inequality to the Burkholder-Gundy inequality]{From the Littlewood-Paley-Stein inequality to the Burkholder-Gundy inequality}


\author[Zhendong  Xu]{Zhendong Xu}
\address{
	Laboratoire de Math{\'e}matiques, Universit{\'e} de Bourgogne Franche-Comt{\'e}, 25030 Besan\c{c}on Cedex, France}
\email{xu.zhendong@univ-fcomte.fr}

\author[Hao Zhang]{Hao Zhang}
\address{
	Laboratoire de Math{\'e}matiques, Universit{\'e} de Bourgogne Franche-Comt{\'e}, 25030 Besan\c{c}on Cedex, France}
\email{hao.zhang@univ-fcomte.fr}
\date{}
\maketitle

\begin{abstract}
	Let  $\{\mathsf{T}_t\}_{t>0}$ be a symmetric diffusion semigroup on a $\sigma$-finite measure space $(\Omega, \mathscr{A}, \mu)$ and $G^{\mathsf{T}}$ the associated  Littlewood-Paley $g$-function operator:
	$$G^{\mathsf{T}}(f)=\Big(\int_0^\infty \left|t\frac{\partial}{\partial t} \mathsf{T}_t(f)\right|^2\frac{\mathrm{d}t}{t}\Big)^{\frac12}.$$
	The classical Littlewood-Paley-Stein inequality asserts that for any $1<p<\infty$ there exist two positive constants $\mathsf{L}^{\mathsf{T}}_{p}$ and $\mathsf{S}^{\mathsf{T}}_{ p}$ such that 
	$$
	\big(\mathsf{L}^{\mathsf{T}}_{ p}\big)^{-1}\big\|f-\mathrm{F}(f)\big\|_{p}\le \big\|G^{\mathsf{T}}(f)\big\|_{p}
	\le  \mathsf{S}^{\mathsf{T}}_{p}\big\|f\big\|_{p}\,,\quad \forall f\in L_p(\Omega),
	$$
	 where $\mathrm{F}$ is the projection from $L_p(\Omega)$ onto the fixed point subspace of  $\{\mathsf{T}_t\}_{t>0}$ of $L_p(\Omega)$.
	 
	 Recently, Xu proved that $	\mathsf{L}^{\mathsf{T}}_{ p}\lesssim p$ as $p\rightarrow\8$, and raised the problem abut the optimal order of $	\mathsf{L}^{\mathsf{T}}_{ p}$ as $p\rightarrow\8$. We solve Xu's open problem by showing that this upper estimate of $\mathsf{L}^{\mathsf{T}}_{ p}$ is in fact optimal. Our argument is based on the construction of a special symmetric diffusion semigroup associated to any given martingale such that its square function $G^{\mathsf{T}}(f)$ for any $f\in L_p(\Omega)$ is pointwise comparable with the martingale square function of $f$. Our method also extends to the vector-valued and noncommutative setting.
\end{abstract}


\section{Introduction}\label{Introduction}
  This article follows the recent investigation of the Littlewood-Paley-Stein theory carried out by Xu. In his paper \cite{Xu1}, Xu has developed a new powerful method to study the vector-valued Littlewood-Paley-Stein theory for semigroups. One interesting benefit of this new method is the fact that it often yields the optimal orders of the relevant best constants. 
  
  Let $\{\mathsf{T}_t\}_{t> 0}$ be  a symmetric diffusion semigroup on a measure space $(\Omega, \mathscr{A}, \mu)$ and $\{\mathsf{P}_t\}_{t>0}$ its subordinated Poisson semigroup defined as follows
  $$ \mathsf{P}_{t}(f)=\frac{1}{\sqrt{\pi}} \int_{0}^{\infty} \frac{e^{-s}}{\sqrt{s}} \sfT_{\frac{t^{2}}{4 s}}(f) d s. $$
  Let $G^\mathsf{P}(f)$ denote the Littlewood-Paley $g$-function of $f\in L_p(\Omega)$ defined by
  \begin{equation}\label{e1.2}
  G^\mathsf{P}(f)=\left(\int_{0}^{\infty}\left|t \frac{\partial}{\partial t} \mathsf{P}_{t}(f)\right|^{2} \frac{d t}{t}\right)^{\frac{1}{2}} . 
  \end{equation}
  Then Stein's celebrated Littlewood-Paley inequality states that for every $1<p<\infty$ there exist two positive constants $\mathsf{L}^{\mathsf{P}}_{p}$ and $\mathsf{S}^{\mathsf{P}}_{ p}$ such that 
  $$
  \big(\mathsf{L}^{\mathsf{P}}_{ p}\big)^{-1}\big\|f-\mathrm{F}(f)\big\|_{p}\le \big\|G^{\mathsf{P}}(f)\big\|_{p}
  \le  \mathsf{S}^{\mathsf{P}}_{p}\big\|f\big\|_{p}\,,\quad \ \forall f\in L_p(\Omega),
  $$
  where $\mathrm{F}$ is the projection from $L_p(\Omega)$ onto the fixed point subspace of  $\{\mathsf{P}_t\}_{t>0}$ of $L_p(\Omega)$.
  
  \
  
  Note that Stein's proof does not give the optimal orders of the constants $\mathsf{L}^{\mathsf{P}}_{ p}$ and $\mathsf{S}^{\mathsf{P}}_{p}$. Although the principal objective of \cite{Xu1} concerns the vector-valued case, Xu has proved the following estimates on $\mathsf{L}^{\mathsf{P}}_{ p}$ and $\mathsf{S}^{\mathsf{P}}_{p}$:
  $$ \mathsf{L}^{\mathsf{P}}_p\lesssim \max\{p, p'^{\frac{1}{2}}\} \ \ \ \text{and} \ \ \   \mathsf{S}^{\mathsf{P}}_p\lesssim \max\{p^{\frac{1}{2}}, p'\}. $$
  Moreover, he has shown that the estimate above on $\mathsf{S}^{\mathsf{P}}_{p}$ is optimal as $p\rightarrow 1$ and $p\rightarrow\8$ since it is so already for the classical Poisson semigroup on $\mathbb{R}$.
  
  \
  
  However, the corresponding problem for $\mathsf{L}^{\mathsf{P}}_{ p}$ is left open in \cite{Xu1}. Note that when $\mathsf{P}$ is the classical Poisson semigroup on $\mathbb{R}^d$, one has $\mathsf{L}^{\mathsf{P}}_{ p}\approx 1$ as $p\rightarrow 1$ (see \cite{Xu4}).
  
  In this article, we show that Xu's above estimate on $\mathsf{L}^{\mathsf{P}}_{ p}$ is in fact optimal as $p\rightarrow\8$, namely,  $\mathsf{L}^{\mathsf{P}}_p\gtrsim p$ as $p\rightarrow\8$. In fact, we will
  prove the stronger inequality $\mathsf{L}^{\mathsf{T}}_p\gtrsim p$ as $p\rightarrow\8$ for a symmetric diffusion semigroup $\{\sfT_t\}_{t>0}$. Thus it only remains to determine the optimal order of  $\mathsf{L}^{\mathsf{P}}_p$ as $p\rightarrow1$.
  
  Our proof relies on the deep relationship between martingale inequalities and Littlewood-Paley-Stein inequalities inspired by a result of Neveu. He establishes a profound link between the martingale theory and the ergodic theory in \cite{NJ}. Indeed, Neveu shows that Doob's maximal inequality for martingale can be obtained from Dunford-Schwartz's maximal ergodic theorem. We will prove the analogous result for the square function.
  
    \  
    
    We recall the definition of symmetric diffusion semigroups in Stein's sense \cite[Chapter 3, Section 1]{St}. Let $(\Omega, \mathscr{A}, \mu)$ be a $\sigma$-finite measure space. $\left\{\mathsf{T}_{t}\right\}_{t>0}$ is called a symmetric
	diffusion semigroup on $(\Omega, \mathscr{A}, \mu)$ if $\left\{\mathsf{T}_{t}\right\}_{t>0}$ satisfies the following conditions: \\
	(a) $\mathsf{T}_{t}$ is a contraction on $L_{p}(\Omega)$ for every $1 \leq p \leq \infty$;\\
	(b) $\mathsf{T}_{t} \mathsf{T}_{s}=\mathsf{T}_{t+s}$ for positive $t$ and $s$;\\
	(c) $\lim_{t \rightarrow 0^+} \mathsf{T}_{t}(f)=f$ in $L_{2}(\Omega)$ 
	for every $f \in L_{2}(\Omega)$;\\
	(d) $\mathsf{T}_{t}$ is positive (i.e. positivity preserving);\\
	(e) $\mathsf{T}_{t}$ is selfadjoint on $L_{2}(\Omega)$;\\
	(f) $\mathsf{T}_{t}(1)=1$.
	
Recall that any symmetric diffusion semigroup $\left\{\mathsf{T}_{t}\right\}_{t>0}$ is analytic on $L_p(\Omega)$  for $1<p<\infty$ which implies that $\frac{\partial}{\partial t} \mathsf{T}_{t}(f)$ is well-defined (see \cite[Chapter 3, Section 2]{St}). As a result, we define $G^{\mathsf{T}}(f)$ by
$$  G^\mathsf{T}(f)=\left(\int_{0}^{\infty}\left|t \frac{\partial}{\partial t} \mathsf{T}_{t}(f)\right|^{2} \frac{d t}{t}\right)^{\frac{1}{2}} . $$
We often call $G^\mathsf{T}(f)$ the Littlewood-Paley $g$-function of $f$, or call  $G^\mathsf{T}(f)$ the square function for semigroups for simplicity.
	
	\
	
	 We will  use the following convention:  $A\lesssim B$ (resp. $A\lesssim_\a B$) means that $A\le C B$ (resp. $A\le C_\a B$) for some absolute positive constant $C$ (resp. a positive constant $C_\a$ depending only on a parameter $\a$).  $A\approx B$ or  $A\approx_\a B$ means that these inequalities as well as their inverses hold.
	 
	 \
	
	It is a classical fact that the orthogonal projection $F$ from $L_{2}(\Omega)$ onto the fixed point subspace of $\left\{\mathsf{T}_{t}\right\}_{t>0}$ extends to a contractive projection on $L_{p}(\Omega)$ for every $1 \leq p \leq \infty$. Then $\mathrm{F}$ is also positive and $\mathrm{F}\left(L_{p}(\Omega)\right)$ is the fixed point subspace of $\left\{\mathsf{T}_{t}\right\}_{t>0}$ on $L_{p}(\Omega)$. Stein's celebrated extension of the classical Littlewood-Paley inequality \cite{St} asserts that for every symmetric diffusion semigroup
	$\left\{\mathsf{T}_{t}\right\}_{t>0}$ and every $1<p<\infty$
	\begin{equation}\label{1.1}
	\|f-\mathrm{F}(f)\|_{L_{p}(\Omega)} \approx_{p}\left\|G^\mathsf{T}(f)\right\|_{L_{p}(\Omega)}, \quad \ \forall f \in L_{p}(\Omega).
	\end{equation}
	
	Stein proves the above inequality (\ref{1.1}) by virtue of Burkholder-Gundy's martingale inequality and complex interpolation. He establishes a close connection between semigroup theory and martingale theory and studies semigroups through martingale theory via Rota's dilation. We will proceed in a reverse way by studying martingales by semigroups.
	
	\
	
	We consider the square function in the martingale setting. Throughout this article we will work with a fixed probability
	space $(\Omega, \mathscr{A}, \mu)$ and a sequence of $\sigma$-algebras
	$$ \mathscr{A}_1\subset \mathscr{A}_2\subset\dots\subset \mathscr{A}_n\subset \dots \subset \mathscr{A} $$
	such that $\cup_{n=1}^{\infty}\mathscr{A}_n=\mathscr{A}$.
	For a random variable $f\in L_{1}(\Omega, \mathscr{A}, \mu)$ we will set $n\in \mathbb{N}^* $
	$$
    E_n(f)=E\left(f \mid{\mathscr{A}}_{n}\right) , \quad d E_n(f)=E_n(f)-E_{n-1}(f),
	$$
	where $E_0=0$ as convention.
	The square function of the martingale $\{E_n\}_{n\geq 1}$ is 
	$$ S(f)=\left(\sum_{n=1}^{\infty}|d E_n(f)|^2\right)^\frac{1}{2}. $$
	 See \cite{GA} for more information on martingale inequalities.
	 
	\
	
	We introduce the most important operator of the article associated to martingales. Given a strictly increasing sequence $\{a_n\}_{n\geq 0}$ with $a_0=0$ and $\lim\limits_{n\rightarrow\infty}a_n=1$, define
	\begin{equation}\label{5.1}
	T=\sum\limits_{n= 1}^{\infty}(a_n-a_{n-1}) E_n.
	\end{equation}
	 Then $T$ is a positive contraction on $L_p(\Omega, \mathscr{A}, \mu)$ for $1\leq p\leq \infty$. In particular, for $p=2$, $T$ is positive and selfadjoint on $L_2(\Omega, \mathscr{A}, \mu)$. We use $T$ to produce a symmetric diffusion semigroup.
	
	\
	
	\begin{thm}\label{thm1.1}
		If we denote by $T^t$ the operator obtained by continuous functional calculus, then for $t>0$ 
		\begin{equation}\label{ee3}
		T^t=\sum\limits_{n= 1}^{\infty}(1-a_{n-1})^td {E}_n=\sum\limits_{n= 1}^{\infty}\left[(1-a_{n-1})^t-(1-a_n)^t\right]{E}_n,
		\end{equation}
		and $T^t$ is a positive contraction for any $t>0$ on $L_p(\Omega, \mathscr{A}, \mu)\ (1\leq p\leq \8)$. Moreover, $\{T^t\}_{t>0}$ is a symmetric diffusion semigroup.
		\end{thm}

		The above semigroup $\{T^t\}_{t>0}$ is exactly our desired semigroup generated by martingale. Then we select some specific sequence of $\{a_n\}_{n\geq 0}$ and use this particular semigroup to obtain the best order of $\mathsf{L}^\mathsf{T}_p$. Below is one of our principal theorem, and it is the square function analogue of Neveu's result \cite{NJ}. Note that $\mathrm{F}=E_1$.
	
\begin{thm}\label{Theorem A}
		If $ a_{n}=1-\mathrm{e}^{-16^{n+1}}$ for $n\geq 1$, then the semigroup $\{T^t\}_{t>0}$ defined in Theorem \ref{thm1.1} satisfies
		$$	\sqrt{\dfrac{7}{60}}\ S(f-\mathrm{F}(f))\leq G^T(f)\leq \sqrt{\dfrac{23}{60}}\ S(f-\mathrm{F}(f)), \ \ \ \forall f\in L_1(\Omega, \mathscr{A}, \mu). $$
	    Consequently, for $1\leq p\leq \infty$
		$$ \sqrt{\dfrac{7}{60}}\ \| S(f-\mathrm{F}(f))\|_p\leq \|G^T(f)\|_p\leq  \sqrt{\dfrac{23}{60}}\ \| S(f-\mathrm{F}(f))\|_p, \ \ \ \forall f\in L_p(\Omega, \mathscr{A}, \mu).  $$
		\end{thm}
		\
		
\begin{rem}
	The choice of $\{a_n\}_{n\geq 1}$ is not unique. There are many other sequences of $\{a_n\}_{n\geq 1}$ satisfying the above pointwise inequalities with different constants, of course.
	\end{rem}
		
 Theorem \ref{Theorem A} shows that for any martingale, there exists a symmetric diffusion semigroup such that their corresponding square functions are equivalent. In this way, we can use the Littlewood-Paley-Stein inequality (\ref{1.1}) to show the analogous inequality for martingale square function. This means that the Burkholder-Gundy square function inequality for martingales can be deduced from the Littlewood-Paley-Stein inequality for semigroups. In this way, the optimal constants in the martingale square function inequality can be applied to the setting of semigroups.  

\

 Recently, Xu has shown the vector-valued Littlewood-Paley-Stein inequality for analytic symmetric diffusion semigroups via holomorphic functional calculus (see \cite{Xu1}). His method is optimal in the sense that it offers the optimal orders of the best constants. For the martingale cotype case, he shows that the classical Poisson semigroup meets the optimal orders (see \cite[Proposition 8.5]{Xu1}). Unfortunately, the optimal order of the reverse Littlewood-Paley-Stein inequality was left unsolved in \cite{Xu1}. In the following, we will utilize Theorem \ref{Theorem A} to give the optimal order of $\mathsf{L}^\mathsf{T}_p$ as $p\rightarrow\8$.

\

To this end, we recall the following Burkholder-Gundy inequality.
	\begin{thm}[{\cite[Theorem 3.1]{Bu}}] \label{thm1.5}
 Let $1<p<\infty$ and $p^*=\max\{p, \frac{p}{p-1}\}$. Then
$$
\left({p}^{*}-1\right)^{-1}\|S({f})\|_{{p}} \leq\|{f}\|_{{p}} \leq\left({p}^{*}-1\right)\|S({f})\|_{{p}}.
$$
In particular,
$$
\begin{array}{l}
\|f\|_{{p}} \geq({p}-1)\|S(f)\|_{{p}}\ \text { if } \quad 1<{p} \leq 2, \\
\|f\|_{{p}} \leq({p}-1)\|S(f)\|_{{p}}\ \text { if } \quad 2 \leq {p}<\infty.
\end{array}
$$
Moreover, the constant $p-1$ is best possible.
\end{thm}

Using Theorem \ref{Theorem A} and Theorem \ref{thm1.5}, we immediately get the following corollary, which solves a problem left open by Xu in \cite{Xu1}.

	\begin{corollary}\label{Theorem B} 
	For any symmetric diffusion semigroup $\{\mathsf{T}_{t}\}_{t>0}$, and for $p\geq 2$
	\begin{equation}\label{equ1.2}
	 \|f-\mathrm{F}(f)\|_p\lesssim p\|G^\mathsf{T}(f)\|_p, \quad \forall f \in L_{p}(\Omega),
	 \end{equation}
	and $p$ is the optimal order of $\mathsf{L}^\mathsf{T}_p$.
	\end{corollary}
	
	\begin{proof}[Proof of Corollary \ref{Theorem B}]
		Indeed, (\ref{equ1.2}) has already been proved in \cite[Theorem 8.1]{Xu1}. Then it suffices to show that $p$ is the optimal order of $\mathsf{L}^\mathsf{T}_p$, which is an immediate consequence of Theorem \ref{Theorem A} and Theorem \ref{thm1.5}.
	\end{proof}

\

Theorem \ref{Theorem A} can be extended to the vector-valued and noncommutative settings. We first consider the vector-valued case. Given a Banach space $X$, let $L_{p}(\Omega ; X)$ denote the $L_{p}$-space of strongly measurable $p$-integrable functions from $\Omega$ to $X .$ It is a well-known elementary fact that if $T$ is a positive bounded operator on $L_{p}(\Omega)$ with $1 \leq p \leq \infty$, then $T \otimes \operatorname{Id}_{X}$ is bounded on $L_{p}(\Omega ; X)$ with the same norm. For notational convenience, throughout this article, we will denote $T \otimes \operatorname{Id}_{X}$ by $T$ too. Let again $\mathrm{F}$ be the projection from $L_{p}(\Omega; X)$ onto the fixed point subspace of $\left\{T_{t}\right\}_{t>0}$. Thus $\left\{T_{t}\right\}_{t>0}$ is also a semigroup of contractions on $L_{p}(\Omega ; X)$ for any Banach space $X$ with $\mathrm{F}\left(L_{p}(\Omega ; X)\right)$ as its fixed point subspace.

\

For an $X$-valued $L_p$-martingale $\{E_n\}_{n\geq 1}$, define the corresponding $q$-variant of the martingale square function for $1\leq q<\8$ as follows
$$ S_q(f)=\left(\sum_{n=1}^{\infty}\|d E_n(f)\|_X^q\right)^\frac{1}{q}, \quad \forall f\in L_p(\Omega; X). $$
Given a strictly increasing sequence $ \{ a_n \}_{n\geqslant 0} $ with $ a_0 = 0 $ and $ \lim\limits_{n \to \infty} a_n = 1  $ as before, let $T^t$ be defined by (\ref{5.1}). 

Then $\left\{T^{t}\right\}_{t>0}$ extended to $L_{p}(\Omega ; X)$ remains to be a strongly continuous semigroup of contractions on $L_{p}(\Omega ; X)$, and $(\ref{ee3})$ remains valid in the $X$-valued case. Define for $q\geqslant 1$
$$
G_q^T(f) =\sk{ \int_{0}^{\infty} \norm{ t \pdt T^t(f) }_X^q \frac{\d t}{t} }^{\frac{1}{q}}, \quad
\forall f \in L_p(\Omega ; X) .
$$

\ 

The following theorem is the vector-valued analogue of Theorem \ref{Theorem A}:
\begin{thm}\label{thm2.1}
	Let $ X $ be a Banach space and $1\leq p,\ q<\8$. Then there exist a sequence $\{a_n\}_{n\geq 0}$ and two universal positive constants, $ c $ and $ C $ such that
	\begin{equation}\label{Gp-Sp}
	c S_q\sk{ f-{\rm F}(f) } \leqslant G_q^T(f) \leqslant C S_q\sk{ f-{\rm F}(f) }.
	\end{equation}
	for any $ f \in L_p(\Omega ; X) $.
\end{thm}

We now turn to the noncommutative setting. We introduce the definitions of square functions for noncommutative martingales and noncommutative semigroups.

\

 Let $\mathcal{M}$ denote a finite von Neumann algebra equipped with a normal faithful normalized trace $\tau$, and $\left(\mathcal{M}_{n}\right)_{n \geq 1}$ an increasing filtration of von Neumann subalgebras of $\mathcal{M}$ whose union is $\mathrm{w}^{*}$-dense in $\mathcal{M}$. For $1 \leq p \leq \infty$ we denote by $L_{p}(\mathcal{M}, \tau)$, or simply $L_{p}(\mathcal{M})$ the usual noncommutative $L_{p}$-space associated with $(\mathcal{M}, \tau)$. As usual, $L_{p}\left(\mathcal{M}_{n}\right)=L_{p}\left(\mathcal{M}_{n},\left.\tau\right|_{\mathcal{M}_{n}}\right)$ is naturally identified as a subspace of $L_{p}(\mathcal{M})$. It is well-known that there exists a unique normal faithful conditional expectation $\mathcal{E}_{n}$ from $\mathcal{M}$ onto $\mathcal{M}_{n}$ such that $\tau \circ \mathcal{E}_{n}=\tau .$ Moreover, $\mathcal{E}_{n}$ extends to a contractive projection from $L_{p}(\mathcal{M})$ onto $L_{p}\left(\mathcal{M}_{n}\right)$, for every $1 \leq p<\infty$, which is still denoted by $\mathcal{E}_{n}$. Similarly, we denote by $d\mathcal{E}_n=\mathcal{E}_{n}-\mathcal{E}_{n-1}$ for $n\geq 1$ the martingale difference with $\mathcal{E}_0=0$ as convention.

A noncommutative martingale with respect to $\left(\mathcal{M}_{n}\right)_{n\geq 1}$ is a sequence $x=\left(x_{n}\right)_{n \geq 1}$ in $L_{1}(\mathcal{M})$ such that
$$
x_{n}=\mathcal{E}_{n}\left(x_{n+1}\right), \quad \forall n \geq 1.
$$
The difference sequence of $x$ is $d x=\left(d x_{n}\right)_{n \geq 1}$, where $d x_{n}=x_{n}-x_{n-1}$ (with $x_{0}=0$ by convention). Then we define $L_{p}$-martingales and bounded $L_{p}$-martingales, as usual.
In the sequel, we will fix $\left(\mathcal{M}, \tau\right)$ and $\left(\mathcal{M}_{n}\right)_{n \geq 1}$ as above and all noncommutative martingales will be with respect to $\left(\mathcal{M}_{n}\right)_{n \geq 1}$.

 Recall that $|\cdot|$ stands for the usual (right) modulus of operators, i.e. $|a|=$ $\left(a^{*} a\right)^\frac{1}{2}$. Define the column and row square functions respectively for $x\in L_1(\M)$
 $$   S_c(x)= \left(\sum_{n \geq 1}\left|d x_{n}\right|^{2}\right)^{\frac{1}{2}} \ \ \text{and} \ \ S_r(x)= \left(\sum_{n \geq 1}\left|d x_{n}^*\right|^{2}\right)^{\frac{1}{2}}. $$
We refer to \cite{JX}, \cite{PX1}, \cite{PX2} and \cite{PX3} for more information.

\

Let $\left(\sfT_{t}\right)_{t > 0}$ be a semigroup of operators on $\mathcal{M}$. We say that $\left(\sfT_{t}\right)_{t > 0}$ is a noncommutative symmetric diffusion semigroup if it satisfies the following conditions:\\
$(\mathsf{a})$ each $\sfT_{t}: \mathcal{M} \rightarrow \mathcal{M}$ is a unital, normal and completely positive; \\
$(\mathsf{b})$ for any $x \in \mathcal{M}, \sfT_{t}(x) \rightarrow x$ in the $w^{*}$-topology of $\mathcal{M}$ when $t \rightarrow 0^{+}$; \\
$(\mathsf{c})$ each $\sfT_t: \mathcal{M} \rightarrow \mathcal{M}$ is selfadjoint. Namely, for any $x, y\in\M$
 $$  \tau(\sfT_{t}(x)y)=\tau(x\sfT_{t}(y)). $$
$(\mathsf{d})$ the extension of each $\sfT_{t}: \mathcal{M} \rightarrow \mathcal{M}$ from $L_p(\M)$ to $L_p(\M)$ for $1\leq p<\8 $ is completely contractive.\\

 It is well-known that such a semigroup extends to a semigroup of contractions on $L_{p}(\mathcal{M})$ for any $1 \leq p<\infty$, and that $\left(\sfT_{t}\right)_{t > 0}$ is a selfadjoint semigroup on $L_{2}(\mathcal{M})$. Moreover, $\left(\sfT_{t}\right)_{t > 0}$ is strongly continuous on $L_{p}(\mathcal{M})$ for any $1 \leq p<\infty$.

Define
$$  G^\sfT_c(x)=\left( \int_{0}^{\8} t\left|\dfrac{\partial}{\partial t}\sfT_t(x)\right|^2\mathrm{d}t\right) ^\frac{1}{2}  $$
and
$$  G^\sfT_r(x)=\left( \int_{0}^{\8} t\left|\dfrac{\partial}{\partial t}\sfT_t(x)^*\right|^2\mathrm{d}t\right) ^\frac{1}{2}. $$
We call $G^\sfT_c$ and $G^\sfT_r$ the column and row square functions, respectively. We refer the reader to \cite{JLX} for more information.
\

Given a strictly increasing sequence $ \{ a_n \}_{n\geqslant 0} $ with $ a_0 = 0 $ and $ \lim\limits_{n \to \infty} a_n = 1  $ as before, analogously we define the mapping $T^t \ (t>0)$ as follows
\begin{equation}\label{eee3}
 T^t=\sum\limits_{n= 1}^{\infty}(1-a_{n-1})^td \mathcal{E}_n=\sum\limits_{n= 1}^{\infty}\left[(1-a_{n-1})^t-(1-a_n)^t\right]\mathcal{E}_n. 
 \end{equation}
 Note that $\mathrm{F}=\mathcal{E}_1$ as well. 

\begin{thm}\label{nthm}
If $ a_{n}=1-\mathrm{e}^{-16^{n+1}}$ for $n\geq 1$, then $\{T^t\}_{t>0}$ defined in (\ref{eee3}) is a noncommutative symmetric diffusion semigroup. Moreover, it satisfies for $x\in L_1(\M)$
\begin{equation}\label{e1.7.1}
 \sqrt{\dfrac{7}{60}} \ S_c(x-\mathrm{F}(x)) \leq G^T_c(x)\leq \sqrt{\dfrac{23}{60}} \ S_c(x-\mathrm{F}(x))
 \end{equation}
and similarly
\begin{equation}\label{e1,7.2}
\sqrt{\dfrac{7}{60}} \ S_r(x-\mathrm{F}(x)) \leq G^T_r(x)\leq \sqrt{\dfrac{23}{60}} \ S_r(x-\mathrm{F}(x)).
 \end{equation}
\end{thm}

This article is organized as follows. The next section is devoted to the proofs of Theorem \ref{thm1.1} and Theorem \ref{Theorem A}. Theorem \ref{thm2.1} and Theorem \ref{nthm} are proved respectively in the third section and last section.
	\bigskip
	
	\section{Proofs of Theorem \ref{thm1.1} and Theorem \ref{Theorem A}}
We will need the following elementary lemmas.
	
	\begin{lemma}\label{lem2.1}
		If $S$ is a positive injective continuous operator on a Hilbert space $\mathcal{H}$, then for $x\in \mathcal{H}$
		$$ \lim\limits_{t\rightarrow0^+}\|S^t(x)-x\|=0. $$
		\end{lemma}
	\begin{proof}
		By the spectral decomposition, there exists a resolution of the identity $\left\lbrace E_\lambda \right\rbrace_{0\leq \lambda\leq \|S\|}$ such that
		$$ S=\int_{0}^{\|S\|}\lambda \mathrm{d}E_\lambda. $$
		Since $S$ is injective, $E_{0}=0$. Hence for $x\in \mathcal{H}$
		$$ x=\int_{(0, \|S\|]} \mathrm{d}E_\lambda(x). $$
		Thus
		\begin{equation*}
		\begin{aligned} 
		 \|S^t(x)-x\|^2&=\left\|\int_{(0, \|S\|]}(\lambda^t-1) \mathrm{d}E_\lambda(x)\right\|^2 \\
		 &=\int_{(0, \|S\|]}|\lambda^t-1|^2\mathrm{d}\langle E_\lambda(x), x\rangle.
		 	 \end{aligned}
		 \end{equation*}
		By the Lebesgue dominated convergence theorem and the fact that $|\lambda^t-1|\rightarrow 0$ as $t\rightarrow0^+$ for $\lambda\neq 0$, we deduce the desired limit.
		\end{proof}
	
	\begin{lem}[Gershgorin circle theorem]\label{lem2.2}
		Let $A=(a_{i, j})_{1\leq i, j\leq n}$ be a complex $n \times n$ matrix. If $\lambda$ is an eigenvalue of $A$, then
		$$ \lambda\in \bigcup_{k=1}^n D(a_{kk}, \sum\limits_{j=1, j\neq k}^n |a_{kj}|), $$
		where $D(a_{kk}, \sum\limits_{j=1, j\neq k}^n |a_{kj}|)$ is the closed disc of the complex plane centered at $a_{kk}$ with radius $\sum\limits_{j=1, j\neq k}^n |a_{kj}|$.
	\end{lem}
	\begin{proof}
		Assume that $\lambda$ is an eigenvalue of $A$. Then there exists a nonzero vector $x=(x_1, \cdots, x_n)$ such that $Ax=\lambda x$. Choose $i$ such that $|x_i|=\text{max}_{1\leq j\leq n}|x_j|>0$. Since
		$$ \sum_{j=1}^{n}a_{ij}x_j=\lambda x_i, $$
		we obtain by the triangle inequality
		$$ |\lambda-a_{ii}|=\left|\sum\limits_{j=1, j\neq i}^n a_{ij}\dfrac{x_j}{x_i}\right| \leq \sum\limits_{j=1, j\neq i}^n |a_{ij}|, $$
		which implies the desired result.
	\end{proof}
	
\begin{proof}[Proof of Theorem \ref{thm1.1}]
    It is easy to prove the first part of Theorem \ref{thm1.1}. Indeed, we can rewrite $T$ as
	$$ T=\sum\limits_{n= 1}^{\infty}(1-a_{n-1})dE_n, $$
	since
	\begin{equation*}
	\begin{aligned} 
	T&=\lim\limits_{n\rightarrow\infty}\sum\limits_{k=1}^n(a_k-a_{k-1}) E_k=	\lim\limits_{n\rightarrow\infty}\sum\limits_{k=1}^n(a_k-a_{k-1}) \sum\limits_{i=1}^{k}d E_i\\
	&=\lim\limits_{n\rightarrow\infty}\sum\limits_{i=1}^nd E_i\sum\limits_{k=i}^{n}(a_k-a_{k-1})=\lim\limits_{n\rightarrow\infty}\sum\limits_{i=1}^n(a_n-a_{i-1})d E_i\\
	&=\lim\limits_{n\rightarrow\infty}\left[\sum\limits_{i=1}^n(1-a_{i-1})d E_i-(1-a_n)E_n\right]\\
	&=\sum\limits_{n= 1}^{\infty}(1-a_{n-1})d E_n.
	\end{aligned}
	\end{equation*}
	Then by the continuous functional calculus and $d E_p d E_q=0$ for $p\neq q$, we obtain
	\begin{equation}\label{ee1}
	T^t=\sum\limits_{n= 1}^{\infty}(1-a_{n-1})^td E_n , \quad t>0. 
	\end{equation}
	Hence we conclude that $T^t$ is a positive contraction for any $t>0$ on $L_p(\Omega, \mathscr{A}, \mu)$ ($1\leq p\leq \8$). Actually, for $t>0$
		\begin{equation}\label{ee2}
	 T^t=\sum\limits_{n= 1}^{\infty}\left[(1-a_{n-1})^t-(1-a_n)^t\right]E_n. 
		\end{equation}
	
	It suffices to show that $\{T^t\}_{t>0}$ is a symmetric diffusion semigroup. $T^t$ satisfies the conditions $(a)$ and $(d)$ since $T^t$ is a positive contraction on $L_p(\Omega, \mathscr{A}, \mu)$ ($1\leq p\leq \8$). The condition $(b)$ is obvious as $T^t$ is generated by $T$ via continuous functional calculus. It remains to prove that $T^t$ satisfies the conditions $(c), (e), (f)$. 
	
	For the condition $(e)$, $T^t$ is selfadjoint on  $L_2(\Omega, \mathscr{A}, \mu)$ since all conditional expectations are selfadjoint. As for the condition $(f)$, 
	$$ T^t(1)=\sum\limits_{n= 1}^{\infty}(1-a_{n-1})^td E_n(1)=(1-a_0)^tdE_1(1)=1.  $$ 
	Note that the conditional expectations are projections on $L_2(\Omega, \mathscr{A}, \mu)$.
	
	In order to show that the semigroup $\lbrace T^t\rbrace_{t>0}$ satisfies the condition $(c)$, it suffices to prove that $T$ is injective on $L_2(\Omega, \mathscr{A}, \mu)$ by Lemma \ref{lem2.1}. This is because for $f\in L_2(\Omega, \mathscr{A}, \mu)$
	\begin{equation}\label{equa2.2}
	\|T(f)\|_2^2=\sum_{n=1}^{\infty}(1-a_{n-1})^2\|d E_n(f)\|_2^2 
	\end{equation}
	and
		\begin{equation}\label{equa2.3} \|f\|_2^2=\lim\limits_{n\rightarrow\infty}\|E_n(f)\|_2^2=\sum_{n=1}^{\infty}\|d E_n(f)\|_2^2. 
			\end{equation}
	\end{proof}

\begin{rem}
	From \eqref{equa2.2} and \eqref{equa2.3}, we can also obtain
	$$ \mathrm{F}(f)=E_1(f), \quad f\in L_p(\Omega, \mathscr{A}, \mu)\ (1\leq p\leq \infty). $$
	Note that
	$$ S(f-\mathrm{F}(f))=\left( \sum_{n=2}^{\infty}|d E_n(f)|^2\right)^\frac{1}{2} .  $$
	\end{rem}

	\begin{proof}[Proof of Theorem \ref{Theorem A}]
	We will calculate the derivative $\frac{\partial}{\partial t} T^t(f)$. To simplify our presentation, we assume that $f\in L_p(\Omega, \mathscr{A}_n, \mu)$ for a fixed $n\geq 2$ since $f$ can be approximated by a sequence of $f_n$ on $L_p(\Omega, \mathscr{A}_n, \mu) (1\leq p<\infty, n\in \mathbb{N}^*)$.
	
	Since $f\in L_p(\Omega, \mathscr{A}_n, \mu)$, $ d E_m(f)=0$ if $m>n$. Then for $t>0$
	$$ T^t(f)=\sum_{k=1}^{n}(1-a_{k-1})^td E_k(f), $$
	which implies
	$$ \frac{\partial}{\partial t} T^t(f)= \sum_{k=1}^{n}(1-a_{k-1})^t \ln(1-a_{k-1}) d E_k(f)=\sum_{k=2}^{n}(1-a_{k-1})^t \ln(1-a_{k-1}) d E_k(f). $$
	Hence
	
	\begin{align}\label{1.2}  
\int_{0}^{\infty}\left|t \frac{\partial}{\partial t} T^{t}(f)\right|^{2} \frac{\mathrm{d} t}{t}& =\int_{0}^{\infty}t\left| \frac{\partial}{\partial t} T^{t}(f)\right|^{2} \mathrm{d}t \notag \\
& =\int_{0}^{\infty} t| \sum_{k=2}^{n}(1-a_{k-1})^t \ln(1-a_{k-1}) d E_k(f) |^{2} \mathrm{d}t \notag \\
& =\sum_{i,j=2}^{n}\int_{0}^{\infty}t(1-a_{i-1})^t (1-a_{j-1})^t \ln(1-a_{i-1})\ln(1-a_{j-1}) d E_i(f) d E_j(f) \mathrm{d}t\notag \\
& =\sum_{i,j=2}^{n}\ln(1-a_{i-1})\ln(1-a_{j-1}) d E_i(f) d E_j(f)\int_{0}^{\infty}t(1-a_{i-1})^t (1-a_{j-1})^t  \mathrm{d}t \notag \\
& =\sum_{i,j=2}^{n}\dfrac{\ln(1-a_{i-1})\ln(1-a_{j-1})}{\ln^2\left[(1-a_{i-1})(1-a_{j-1})\right]} d E_i(f) d E_j(f).
	 \end{align}
We write $y_i=-\ln(1-a_{i-1}) \ (i\geq 2)$. Then $\lbrace y_i\rbrace_{i\geq 2}$ is a strictly increasing positive sequence which tends to infinity as $i\rightarrow\infty$. Let  $B=(b_{ij})_{1\leq i, j\leq n-1}$ be the real symmetric matrix given by
$$ b_{ij}=\dfrac{y_{i+1} y_{j+1}}{(y_{i+1}+y_{j+1})^2} .$$
We need to determine the lower and upper bounds of the eigenvalues of $B$. 

Let  $\lambda$ be an eigenvalue of $B$. By Lemma \ref{lem2.2} we have 
$$  \lambda\in \bigcup_{k=1}^{n-1} D(b_{kk}, \sum\limits_{j=1, j\neq k}^{n-1} |b_{kj}|).$$
However $b_{kk}=\frac{1}{4}$ and $y_i=16^i$ for $i\geq 2$,  thus for $k\neq j$
$$ b_{kj}=\dfrac{y_{k+1}y_{j+1}}{(y_{k+1}+y_{j+1})^2}=\left(\frac{y_{k+1}}{y_{j+1}}+\frac{y_{j+1}}{y_{k+1}}+2\right)^{-1}\leq 16^{-|k-j|}. $$
Thus for every $1\leq k\leq n-1$,
$$ \sum\limits_{j=1, j\neq k}^{n-1} |b_{kj}|\leq \sum\limits_{j=1, j\neq k}^{n-1} 16^{-|k-j|}\leq 2\sum\limits_{i=1}^{\infty}16^{-i}=\dfrac{2}{15}.  $$
We then get
$$ \dfrac{7}{60}\leq \lambda \leq \dfrac{23}{60}. $$
Thus $B$ satisfies ($I$ denoting the identity matrix)
$$ \dfrac{7}{60}I\leq B \leq \dfrac{23}{60}I.  $$
Therefore by (\ref{1.2}), 
$$ \dfrac{7}{60}\sum_{k=2}^{n}|d E_n(f)|^2\leq \int_{0}^{\infty}\left|t \frac{\partial}{\partial t} T^{t}(f)\right|^{2} \frac{\mathrm{d} t}{t}\leq  \dfrac{23}{60}\sum_{k=2}^{n}|d E_n(f)|^2 $$
which implies 
\begin{equation}\label{1.3}
 \sqrt{\dfrac{7}{60}}\ S(f-\mathrm{F}(f))\leq G^T(f)\leq \sqrt{\dfrac{23}{60}}\ S(f-\mathrm{F}(f)). 
 \end{equation}
\end{proof}
 
 \begin{rem}
 	By virtue of the special semigroups induced by martingales in Theorem \ref{thm1.1} and with the help of Theorem \ref{Theorem A}, we get a lower bound of optimal orders of best constants of Littlewood-Paley-Stein inequality for symmetric diffusion semigroups via inequalities of martingale square functions. However, to show that this lower bound given by martingale inequalities is exactly the optimal order of best constants for Littlewood-Paley-Stein inequality  is much more delicate. For more details, see \cite{Xu1}.
 	\end{rem}
 
\bigskip

\section{Proof of Theorem \ref{thm2.1}}
The proof of Theorem \ref{thm2.1} is more complicated than the one of Theorem \ref{Theorem A}. In contrast to Theorem \ref{Theorem A}, the choice of the sequence $\{a_n\}_{n\geq 0}$ is not explicit for Theorem \ref{thm2.1}.

To make our proof easier and neater, in the sequel we denote $ b_n = - \ln (1 - a_n) $ for convenience. Then $ b_0 = 0 $ and $ \lim\limits_{n \to \infty} b_n = +\infty $. Let $1\leq p, q<\8$.  We can also write $G_q^T(f)$ in the following form
\begin{equation}\label{Gp}
G_q^T(f) =\sk{ \int_{0}^{\infty} \norm{\sum_{k = 1}^{\infty} t b_k e^{-t b_k} d E_{k+1}(f) }_X^q \frac{\d t}{t} }^{\frac{1}{q}}, \quad
\forall f \in L_p(\Omega ; X).
\end{equation}
Hence to find a suitable sequence $\{a_n\}_{n\geq 1}$ that satisfies Theorem \ref{thm2.1}, we should search for the existence of the sequence $\{b_k\}_{k\geq 0}$.

Fixing $ M $ sufficiently large, we denote
$$
G_{q, M}^T(f) =\sk{ \int_{0}^{M} \norm{\sum_{k = 1}^{\infty} t b_k e^{-t b_k} d E_{k+1}(f) }_X^q \frac{\d t}{t} }^{\frac{1}{q}}.
$$
Now we are about to prove
\begin{equation}\label{i2.2}
G_q^T(f) \approx_{p, q} S_q\sk{ f-{\rm F}(f) }, \quad \forall f \in L_p(\Omega ; X).
\end{equation}
Indeed, we have a stronger result that the constants in (\ref{i2.2}) are universal.

\begin{proof}[Proof of Theorem \ref{thm2.1}]
	Take a strictly decreasing sequence $ \{ t_n \}_{n\geqslant 0} $ with $\lim\limits_{n\rightarrow\infty}t_n=0$ and $ t_0 = M $,  then
	$$
	G_{q, M}^T(f) =\sk{\sum_{n=1}^{\infty} \int_{t_n}^{t_{n-1}} \norm{\sum_{k = 1}^{\infty} t b_k e^{-t b_k} d E_{k+1}(f) }_X^q \frac{\d t}{t} }^{\frac{1}{q}}.
	$$
	For $n\geq 1$ let
	$$ R_n =\sk{ \int_{t_n}^{t_{n-1}} \norm{\sum_{k = 1}^{\infty} t b_k e^{-t b_k} d E_{k+1}(f) }_X^q \frac{\d t}{t} }^{\frac{1}{q}},$$
and for $ k\geq 1$
	$$ R_{n, k} =\sk{ \int_{t_n}^{t_{n-1}} \norm{ t b_k e^{-t b_k} d E_{k+1}(f) }_X^q \frac{\d t}{t} }^{\frac{1}{q}}. $$
	Considering the interval $ \sk{ t_{n},t_{n-1} } $ equipped with the measure $ \frac{\d t}{t} $, by the triangle inequality, we get
	\begin{align*}
	R_n & = \norm{ \sum_{k = 1}^{\infty} t b_k e^{-t b_k} d E_{k+1}(f)}_{L_q(\sk{ t_{n},t_{n-1} }, \frac{\d t}{t}; X)}\notag \\
	&\geqslant \left|\norm{t b_n e^{-t b_n} d E_{n+1}(f) }_{L_q(\sk{ t_{n},t_{n-1} }, \frac{\d t}{t}; X)} -\sk{ \sum_{k \neq n, k\geq 1} \norm{  t b_n e^{-t b_n} d E_{n+1}(f)}_{L_q(\sk{ t_{n},t_{n-1} }, \frac{\d t}{t}; X)} }\notag \right|\\
	& = \left|R_{n, n} - \sk{ \sum_{k \neq n, k\geq 1} R_{n, k} }\right|.
	\end{align*}
	Thus we deduce
	\begin{align*}
	G_{q, M}^T(f)& = \sk{ \sum_{n=1}^{\infty} R_n^q }^{\frac{1}{q}}  \geqslant \sk{ \sum_{n=1}^{\infty}\jdz{ R_{n, n} - \sk{ \sum_{k \neq n, k\geq 1} R_{n, k} } }^q }^{\frac{1}{q}}.
	\end{align*}
To get the desired control,  we are concerned with making $ \sum\limits_{k \neq n, k\geq 1} R_{n, k}  $ sufficiently small.
	
   We take $ l_k $ and $ m_k $ $(k \geqslant 1)$ such that
	\begin{align*}
	\int_{0}^{l_k} t^{q-1}e^{-t} \d t & = \dfrac{1}{2^{q^2(k+2)}}\Gamma (q),  \\
	\int_{m_k}^{\infty} t^{q-1}e^{-t} \d t & = \dfrac{1}{2^{q^2(k+2)}}\Gamma (q),
	\end{align*}
	where  $ \Gamma(q) = \int_{0}^{\infty} t^{q-1}e^{-t} \d t,\   \text{as } \ q\geq 1 $.
	Denote $ l_0 = m_0 =1 $. It is easy to verify that $ l_k<m_k$ $(k\geq 1)$, $ \lim\limits_{k \to \infty} l_k = 0 $ and $ \lim\limits_{k\to \infty} m_k = +\infty $.
	
	 Then $\{t_k\}_{k\geq 0} $ and $ \{b_k\}_{k\geq 0} $ are just defined by the following formulas
	\begin{align*}
	& t_0 = M,\quad b_0 = 0,\\
	& t_kb_kq = l_k, \quad
	t_{k-1}b_kq = m_k,\qquad k\geqslant 1.
	\end{align*}
	By elementary calculations, it follows that
	\begin{align*}
	t_k & = M \prod_{j=0}^{k} \dfrac{l_j}{m_j}\ \ \ \ \ \ {\rm for}\ k\geqslant 0, \\
	b_k & = \dfrac{1}{Mq} \prod_{j=1}^{k} \dfrac{m_j}{l_{j-1}}\ \ \  {\rm for}\ k\geqslant 1\ {\rm and}\ b_0 =0.
	\end{align*}
	It is obvious that $\{t_k\}_{k\geq 0}$ is decreasing, $\{b_k\}_{k\geq 0}$ is increasing, and
	$$ \lim_{k\rightarrow\infty}t_k=0, \ \ \lim_{k\rightarrow\infty}b_k=\infty. $$
	Then
	\begin{align*}
	|R_{n, n}|^q & = \int_{t_n}^{t_{n-1}} \norm{ t b_n e^{-t b_n} d E_{n+1}(f) }_X^q \frac{\d t}{t}  \\
	& = \int_{t_nb_nq}^{t_{n-1}b_nq} t^{q-1} e^{-t} \d t \frac{\norm{  d E_{n+1}(f) }_X^q}{q^q} \\
	& = \int_{l_n}^{m_n} t^{q-1} e^{-t} \d t  \frac{\norm{  d E_{n+1}(f) }_X^q}{q^q} \\
	& = \sk{ 1 - \frac{2}{2^{q^2(n+2)}} } \frac{\Gamma(q)}{q^q}\norm{  d E_{n+1}(f) }_X^q.
	\end{align*}
	On the other hand, for $1\leq k < n $,
	\begin{align*}
	|R_{n, k}|^q & = \int_{t_nb_kq}^{t_{n-1}b_kq} t^{q-1} e^{-t} \d t \frac{\norm{  d E_{k+1}(f) }_X^q}{q^q} \\
	& \leqslant \int_{0}^{t_{n-1}b_{n-1}q} t^{q-1} e^{-t} \d t \frac{\norm{  d E_{k+1}(f) }_X^q}{q^q} \\
	& = \int_{0}^{l_{n-1}} t^{q-1} e^{-t} \d t \frac{\norm{  d E_{k+1}(f) }_X^q}{q^q}\\
	& = \frac{\Gamma(q)}{2^{q^2(n+1)}q^q}\norm{  d E_{k+1}(f) }_X^q.
	\end{align*}
	Similarly, for $ k > n $
	\begin{align*}
	|R_{n, k}|^q & \leqslant \int_{m_k}^{\infty} t^{q-1} e^{-t} \d t \frac{\norm{  d E_{k+1}(f) }_X^q}{q^q} \\
	& = \frac{\Gamma(q)}{2^{q^2(k+2)}q^q}\norm{  d E_{k+1}(f) }_X^q.
	\end{align*}
	 Let $ r_n,r_{n, k} > 0 $ be defined as follows
	\begin{align*}
	|r_n|^q & = \sk{ 1 - \frac{2}{2^{q^2(n+2)}} } \frac{\Gamma(q)}{q^q}\norm{  d E_{n+1}(f) }_X^q, \\
	|r_{n, k}|^q & = \frac{\Gamma(q)}{2^{q^2(n+1)}q^q}\norm{  d E_{k+1}(f) }_X^q \ {\rm when}\ n>k\geq 1, \\
	|r_{n, k}|^q & = \frac{\Gamma(q)}{2^{q^2(k+2)}q^q}\norm{  d E_{k+1}(f) }_X^q \ {\rm when}\ n<k, \\
	r_{k, k}& = 0.
	\end{align*}
	Denote $ r:= \{ r_n \}_{n\geqslant 1},\ r^{(k)}:= \{ r_{n, k} \}_{n\geqslant 1} $. We estimate the $\ell_q-$norms of $r$ and $r^{(k)}$. By the definition of $r$, we have
	$$ \|r\|_{\ell_q}=\sk{ \sum_{n=1}^{\infty} \sk{ 1 - \frac{2}{2^{q^2(n+2)}} } \frac{\Gamma(q)}{q^q}\norm{  d E_{n+1}(f) }_X^q  }^{\frac{1}{q}}\geq   \left(\sk{ 1 - \frac{2}{2^{3q^2}} } \frac{\Gamma(q)}{q^q}\right)^\frac{1}{q}S_q(f-\mathrm{F}(f)).$$
	It remains to estimate $\left\|\sum\limits_{k\geq 1}^\infty r^{(k)}\right\|_{\ell_q}$. We consider the measure space $(\mathbb{N}^2, 2^{\mathbb{N}^2}, \mu)$ where
	\begin{equation*}
\mu\left( \{(n, k)\} \right)= \left\{\begin{array}{ccc}
\left(\frac{\Gamma(q)}{2^{q^2(n+1)}q^q}\right)^\frac{1}{q} &, \  \text{if} &1\leq k<n\\
0 &, \ \text{if} &k=n\\
\left(\frac{\Gamma(q)}{2^{q^2(k+2)}q^q}\right)^\frac{1}{q} &, \ \text{if} &k>n
\end{array}\right. .
\end{equation*}
It is very easy to verify that $\mu$ is a finite measure. Assume that $\nu$ is the counting measure on $\mathbb{N}$. Define
	\begin{equation*}
	\begin{aligned}
	f : \mathbb{N}^2&\longrightarrow \mathbb{R}_+\\
	(n, k)&\longmapsto \norm{  d E_{k+1}(f) }_X.
	\end{aligned}
	\end{equation*}
Then we deduce by the H\"{o}lder inequality
\begin{align*}
\left\|\sum\limits_{k\geq 1}^\infty r^{(k)}\right\|_{\ell_q}&=\left(\sum_{n=1}^{\infty}\left|\sum_{k=1}^{\infty}r_{n, k}\right|^q\right)^\frac{1}{q}\\
&=\left(\int_{n\in\mathbb{N}}\left| \int_{k\in\mathbb{N}}f(n, k)\d \mu\right|^q\d \nu\right)^\frac{1}{q}\\
&\leq \left( \int_{n\in\mathbb{N}} \left|\int_{k\in\mathbb{N}}f(n, k)^q\d \mu\right| \left( \int_{k\in\mathbb{N}} 1 \d\mu \right)^{q-1} \d \nu\right) ^\frac{1}{q}\\
&\leq \left( \int_{n\in\mathbb{N}} \int_{k\in\mathbb{N}}f(n, k)^q\d\mu\d \nu\right)^\frac{1}{q}\sup_{n\geq 1}\left( \int_{k\in\mathbb{N}} 1 \d\mu \right)^\frac{q-1}{q}.
\end{align*}
However, for $n\geq 1$
\begin{align*}
\left( \int_{k\in\mathbb{N}} 1 \d\mu \right)^\frac{q-1}{q}&\leq \left( n\left(\frac{\Gamma(q)}{2^{q^2(n+1)}q^q}\right)^\frac{1}{q}+\sum_{k>n}\left(\frac{\Gamma(q)}{2^{q^2(k+2)}q^q}\right)^\frac{1}{q}  \right)^\frac{q-1}{q}\\
&=\left( \dfrac{\Gamma(q)}{q^q}\right)^\frac{q-1}{q^2}\left( \dfrac{n}{2^{q(n+1)}}+\sum_{k>n}\dfrac{1}{2^{q(k+2)}} \right)^\frac{q-1}{q}\\
&\leq \left( \dfrac{1}{2}\right)^\frac{q-1}{q} \left( \dfrac{\Gamma(q)}{q^q}\right)^\frac{q-1}{q^2}.
\end{align*}
Moreover, in the same way
\begin{align*}
\left( \int_{n\in\mathbb{N}} \int_{k\in\mathbb{N}}f(n, k)^q\d \mu\d \nu\right)^\frac{1}{q}&=\left( \sum_{k=1}^{\infty}\left( k\left(\frac{\Gamma(q)}{2^{q^2(k+2)}q^q}\right)^\frac{1}{q}+\sum_{n>k}\left(\frac{\Gamma(q)}{2^{q^2(n+1)}q^q}\right)^\frac{1}{q}  \right)\|d E_{k+1}(f)\|_X^q\right)^\frac{1}{q}\\
&\leq \left( \dfrac{1}{2}\right)^\frac{1}{q} \left( \dfrac{\Gamma(q)}{q^q}\right)^\frac{1}{q^2}S_q(f-\mathrm{F}(f)).
\end{align*}
Thus we obtain
$$ \left\|\sum\limits_{k\geq 1}^\infty r^{(k)}\right\|_{\ell_q}\leq \frac{1}{2}\left( \dfrac{\Gamma(q)}{q^q}\right)^\frac{1}{q}S_q(f-\mathrm{F}(f)). $$
	By the triangle inequality again, then it follows
	\begin{align*}
	G_{q, M}^T(f)& \geqslant \sk{ \sum_{n=1}^{\infty}\jdz{ R_{n, n} - \sk{ \sum_{k \neq n, k\geq 1} R_{n, k} } }^q }^{\frac{1}{q}} \geqslant \sk{ \sum_{n=1}^{\infty}\jdz{ r_n - \sk{ \sum_{k \geq 1} r_{n, k} } }^q }^{\frac{1}{q}} \\
	& = \norm{ r - \sk{ \sum_{k = 1}^{\infty} r^{(k)} }  }_{\ell_q} \geqslant \norm{r}_{\ell_q} - \norm{ \sum_{k = 1}^{\infty} r^{(k)} }_{\ell_q}  \\
	& \geq \left(\sk{ 1 - \frac{2}{2^{3q^2}} } \frac{\Gamma(q)}{q^q}\right)^\frac{1}{q}S_q(f-\mathrm{F}(f))-\frac{1}{2}\left( \dfrac{\Gamma(q)}{q^q}\right)^\frac{1}{q}S_q(f-\mathrm{F}(f))\\
	&\geq \left( \dfrac{1}{2}-\frac{2}{2^{3q^2}}\right) \left( \dfrac{\Gamma(q)}{q^q}\right)^\frac{1}{q}S_q(f-\mathrm{F}(f))\\
	&\geq \dfrac{1}{4} \left( \dfrac{\Gamma(q)}{q^q}\right)^\frac{1}{q}S_q(f-\mathrm{F}(f)).
	\end{align*}
Since $\left( \dfrac{\Gamma(q)}{q^q}\right)^\frac{1}{q}\gtrsim 1$ for all $q\geq 1$, it is easy to deduce that
	\begin{equation}\label{ie1}
	G_q^T(f) \geqslant G_{q, M}^T(f) \gtrsim S_q(f - {\rm F}(f)).
	\end{equation}
	
	For the reversing inequality, we have by the convexity of norms
	\begin{align*}
	G_q^T(f)^q & = \int_{0}^{\infty} \norm{\sum_{k = 1}^{\infty} t b_k e^{-t b_k} d E_{k+1}(f) }_X^q \frac{\d t}{t} \\
	& \leqslant  \int_{0}^{\infty} t^{q-1} \sk{\sum_{k=1}^{\infty} b_k e^{-t b_k} }^{q-1} \sum_{k=1}^{\infty} b_k e^{-t b_k}\norm{ d E_{k+1}(f) }_X^q \d t.
	\end{align*}
	To this end, we choose a positive number $ N > 2 $ such that $ b_{k+1}/b_k = m_{k+1}/l_{k} > N $ for all  $k\geq 1$. Then we have $ b_k \leqslant \frac{N}{N-1} (b_k - b_{k-1}) $, so
	$$
	\sum_{k=1}^{\infty} b_k e^{-t b_k} \leqslant \frac{N}{N-1} \sum_{k\geq1}\int_{b_{k-1}}^{b_k} e^{-tx} \d x\leq \frac{N}{N-1}\int_{0}^{\infty} e^{-tx} \d x \leqslant \frac{N}{(N-1)t}.
	$$
	Then
	\begin{align}\label{ie2}
	G_q^T(f)^q & \leqslant \sk{ \frac{N}{N-1} }^{q-1}\sum_{k=1}^{\infty}\mk{ \norm{ d E_{k+1}(f) }_X^q \int_{0}^{\infty} b_k e^{-t b_k} \d t }\notag\\
	& \leqslant \sk{ \frac{N}{N-1} }^{q-1}\sum_{k=1}^{\infty}\norm{ d E_{k+1}(f) }_X^q \notag \\
	& \leqslant \sk{ \frac{N}{N-1} }^{q-1} S_q(f-{\rm F}(f))^q.
	\end{align}
	
	Combine (\ref{ie1}) and (\ref{ie2}), and we finally obtain
	
	\begin{equation}\label{GS}
S_q(f-{\rm F}(f)) \approx G_q^T(f) .
	\end{equation}
	This completes the proof.
\end{proof}

Recall that a Banach space $X$ is of martingale cotype $q$ ($q\geq 2$) if there exists a positive constant $c$ such that every finite $X$-valued $L_{q}$-martingale $\left\{E_{n}(f)\right\}_{n\geq 1}$ satisfies the following inequality
\begin{equation}\label{e2.4}
\mathbb{E} [S_q(f)^q ] \leq c^{q} \sup _{n\geq 1} \mathbb{E}\left\|E_{n}(f)\right\|_{X}^{q},
\end{equation}
where $\mathbb{E}$ denotes the underlying expectation. See \cite{PG1} for more information about martingale cotype. Pisier's famous renorming theorem asserts that $X$ is of martingale cotype $q$ iff $X$ admits an equivalent norm that is  $q$-convex. Besides, he also proved that $X$ is of martingale cotype $q$ iff all $X$-valued Walsh-Paley martingales satisfy (\ref{e2.4}). See \cite{PG} for more details.

\begin{corollary}\label{co2.2}
  Assume $2\leq q<\infty$. If for every symmetric diffusion semigroup $\left\{T_{t}\right\}_{t>0}$ and for every $1<p<\infty$ there exists a constant $c$ such that
	\begin{equation}\label{e2.7}
	\left\|\left(\int_{0}^{\infty}\left\|t \frac{\partial}{\partial t} T_{t}(f)\right\|_{X}^{q} \frac{d t}{t}\right)^{\frac{1}{q}}\right\|_{L_{p}(\Omega)} \leq c\|f\|_{L_{p}(\Omega ; X)}, \quad \forall f \in L_{p}(\Omega ; X),
	\end{equation}
	then $X$ is of martingale cotype $q$.
	\end{corollary}
\begin{proof}
	It is an immediate consequence of Theorem \ref{thm2.1} and the definition of martingale cotype.
	\end{proof}

\begin{rem}
	Corollary \ref{co2.2} was first shown by Xu in \cite[Theorem 3.1]{Xu2}. There he constructed a highly lacunary Fourier series to show that all Walsh-Paley martingales satisfied (\ref{e2.4}) only by using the Poisson semigroup on the unit disk. Hence, from his proof, $X$ is martingale cotype $q$ iff the Poisson semigroup satisfies (\ref{e2.7}).
	\end{rem}

\begin{rem}
	Here we use Theorem \ref{thm2.1} to give an alternative proof of  ``if'' part of \cite[Theorem A]{Xu1} or \cite[Theorem 2.1]{Xu3}. However, our proof needs much additional and stronger conditions on semigroups.
	\end{rem}

\begin{rem}
	From Theorem \ref{thm2.1}, the best constants of martingale inequalities can be dominated by the best constants of  the Littlewood-Paley-Stein inequality. Accordingly, we can use martingale inequities to obtain the Littlewood-Paley-Stein inequality, which is due to Stein \cite[Chapter IV.5]{St}.
	\end{rem}

\bigskip

\section{Proof of Theorem \ref{nthm}}

We end this article with the proof of Theorem \ref{nthm}. Although it is similar to that of Theorem \ref{Theorem A}, we give some details for reader's convenience.

Note that for $1<p<\8$, any noncommutative symmetric diffusion semigroup $\{\sfT_t\}_{t>0}$ is analytic. See \cite{JLX} for details. Then the column and row square functions are well-defined.
\begin{proof}[Proof of Theorem \ref{nthm}]
At first, we show that $\{T^t\}_{t>0}$ is a noncommutative symmetric diffusion semigroup by the same method as Theorem \ref{thm1.1}. More precisely, the condition $(\mathsf{a})$ follows from the fact that all conditional expectations on $\M$ are unital normal and completely positive. As for the condition $(\mathsf{b})$, observe that $T_t(x)\rightarrow x$ for $x\in \M_n\ (\forall n\geq 1)$ as $t\rightarrow0^+$ in the  $w^{*}$-topology of $\M$, and then use the fact that $\cup_{n\geq 1}\M_n$ is  $w^{*}$-dense in $\M$. 

The condition $(\mathsf{c})$ is due to the fact that all conditional expectations on $\M$ are selfadjoint. The condition $(\mathsf{d})$ follows from the fact that any conditional expectation can extend to a completely contractive operator from $L_p(\M)$ to $L_p(\M)$ for $1\leq p<\8$. Hence, we conclude that  $\{T^t\}_{t>0}$ is a noncommutative symmetric diffusion semigroup.

Then it suffices to prove (\ref{e1.7.1}) and (\ref{e1,7.2}). Observe that if we have
    \begin{equation}\label{4.1}
    	  \dfrac{7}{60}S_c(x-\mathrm{F}(x))^2 \leq G^T_c(x)^2\leq \dfrac{23}{60}S_c(x-\mathrm{F}(x))^2,
    \end{equation}
    then the desired inequality follows immediately since the square root function is operator increasing. We only need to consider the case where $x\in L_p(\M_n)$ for a fixed arbitrary $n\in \mathbb{N}$. By the same calculation as (\ref{1.2}), we obtain
    $$
    G^T_c(x)^2=\sum_{i,j=2}^{n}\dfrac{\ln(1-a_{i-1})\ln(1-a_{j-1})}{\ln^2\left[(1-a_{i-1})(1-a_{j-1})\right]} d x_i^* d x_j.$$
    Since $ dx_i^*, dx_j $ and $ d x_i^* d x_j$ $(2\leq i, j\leq n)$ are all measurable operators affiliated to $\M$, they have a common domain denoted by $\mathcal{K}$ which is a dense subspace of $\mathcal{H}$.
    
    For any $\xi\in \mathcal{K}$, 
    \begin{equation*}
    \begin{aligned}
    	\la G^T_c(x)^2(\xi), \xi\ra&=\sum_{i,j=2}^{n}\dfrac{\ln(1-a_{i-1})\ln(1-a_{j-1})}{\ln^2\left[(1-a_{i-1})(1-a_{j-1})\right]} \la d x_i^* d x_j (\xi), \xi\ra\\
    	&=\sum_{i,j=2}^{n}\dfrac{\ln(1-a_{i-1})\ln(1-a_{j-1})}{\ln^2\left[(1-a_{i-1})(1-a_{j-1})\right]} \la d x_j (\xi), d x_i(\xi)\ra,
    	\end{aligned}
    	\end{equation*}
    	which implies that
    	$$   \dfrac{7}{60}\sum_{i = 2}^n \la d x_{i}(\xi), d x_{i}(\xi)\ra  \leq 	\la G^T_c(x)^2(\xi), \xi\ra\leq \dfrac{23}{60}\sum_{i = 2}^n \la d x_{i}(\xi), d x_{i}(\xi)\ra. $$
    	
    	Using the fact that $\la d x_{i}(\xi), d x_{i}(\xi)\ra =\la d x_{i}^* d x_{i}(\xi), (\xi)\ra $, we obtain the desired inequality (\ref{4.1}). As for the row square function, it can be deduced by the same way.
\end{proof}

We introduce the norms in the Hardy spaces of martingales defined in \cite{PX1} and \cite{PX2}. Let $1 \leq p \leq \infty$ and $x=\left(x_{n}\right)_{n \geq 1}$ be an $L_{p}$-martingale. Set, for $p \geq 2$
$$
\|x\|_{\mathcal{H}^{p}}=\max \left\{\left\|S_c(x)\right\|_{p},\left\|S_r(x)\right\|_{p}\right\}
$$
and for $p<2$
$$
\|x\|_{\mathcal{H}^{p}}=\inf \left\{\left\|S_c(y)\right\|_{p}+\left\|S_r(z)\right\|_{p}\right\},
$$
where the infimum runs over all decompositions $x=y+z$ of $x$ as sums of two $L_{p}$ martingales.

\

 Similarly, we introduce the semigroup analogues of these norms. Let $\| \cdot \|_{p, F}$ denote the $p$-norm of semigroup square function defined as follows: for $x\in L_p(\M)$
 $$ \quad\|x\|_{p, F}=\max \left\{\|G_c^T(x)\|_p, \ \|G_r^T(x)\|_p \right\} \ \ \  \text{if} \ \ \ 2 \leq p<\infty, $$
 and
 $$ \|x\|_{p, F}=\inf \left\{ \| G_c^T(y)\|_p+\|G_r^T(z)\|_p : x=y+z \right\}  \ \ \ \text{if} \ \ \ 1 \leq p<2. $$
Note that our definition of $\|x\|_{p, F}$ is different from that of $\|x\|_{p, F}$ in \cite[Chapter 6]{JLX} when $1\leq p<2$.

\begin{rem}
	 It should be noted that when $1\leq p<\8$, we have the same equivalence for the $p$-norm  of their corresponding square functions like the classical setting for the noncommutative symmetric diffusion semigroup in Theorem \ref{nthm}, i.e. 
	$$ \sqrt{\dfrac{7}{60}}\ \|x\|_{\mathcal{H}_p}\leq \|x\|_{p, F}\leq \sqrt{\dfrac{23}{60}}\ \|x\|_{\mathcal{H}_p}, \ \ \ \forall x\in L_p(\M), $$
	which follows from (\ref{e1.7.1}) and (\ref{e1,7.2}).
\end{rem}

\bigskip {\textbf{Acknowledgments.}} It is Professor Quanhua Xu who proposes to us the subject of this article.

	\end{document}